%% file: gaps32011.tex
\documentclass[11pt]{amsart}
\usepackage{amsmath,amsfonts,amssymb,amsthm,graphicx,epsfig,color}
\newcommand{\bb}{\mathbb}
\newcommand{\C}{\bb C}
\newcommand{\h}{\bb H}
\newcommand{\Z}{\bb Z}
\newcommand{\R}{\bb R}
\newcommand{\N}{\bb N}

\newcommand{\Q}{\bb Q}
\newcommand{\M}{\mathcal M}

\newcommand{\T}{\mathcal T}
\newcommand{\hh}{\mathcal H}
\newcommand{\om}{\omega}
\newcommand{\Om}{\Omega}
\newcommand{\vv}{\mathbf v}
\newcommand{\ww}{\mathbf w}

\newcommand{\La}{\Lambda}

\newcommand{\semidirect}{\ltimes}
\newlength{\figboxwidth}             
\newcommand{\makefig}[3]{
        \begin{figure}[htb]
        \refstepcounter{figure}
        \label{#2}
        \begin{center}
                #3~\\
                \smallskip
                Figure \thefigure.  #1
        \end{center}
        \medskip
        \end{figure}
}

\newtheorem{Theorem}{Theorem}
\newtheorem{Cor}[Theorem]{Corollary}
\newtheorem{Prop}[Theorem]{Proposition}
\newtheorem{lemma}[Theorem]{Lemma}
\newtheorem*{lemma*}{Lemma}

\newtheorem*{theorem*}{Theorem}

\numberwithin{equation}{section}
\numberwithin{Theorem}{section}
\begin{document}
\title{The distribution of gaps for Saddle Connection Directions} 
\author{J.~S.~Athreya and J.~Chaika}
\thanks{J.~C. partially supported by NSF grant
    DMS 1004372.} 
\email{jathreya@iilinois.edu}
\email{jonchaika@gmail.com}
\address{\noindent Deptartment of Mathematics, University of Illinois Urbana-Champaign, 1409 W. Green Street, Urbana, IL 61801}
\address{\noindent Department of Mathematics, University of Chicago, 5734 S. University Avenue, Chicago, IL 60637}
\date{Fall 2010}
\begin{abstract} Motivated by the study of billiards in polygons, we prove fine results for the distribution of gaps of directions of saddle connections on translation surfaces. As an application we prove that for almost every holomorphic differential $\omega$ on a Riemann surface of genus $g \geq 2$ the smallest gap between saddle connection directions of length at most a fixed length decays faster than quadratically in the length. We also characterize the exceptional set: the decay rate is not faster than quadratic if and only if $\omega$ is a lattice surface.\end{abstract}
\maketitle

\section{Introduction}\label{sec:intro}

\subsection{Generalized diagonals for rational billiards}\label{subsec:gen:diag} Let $P$ be a Euclidean polygon with angles in $\pi\Q$. We call such a polygon \emph{rational}. A classical dynamical system is given by the idealized motion of a billiard ball on $P$: the (frictionless) motion of a point mass at unit speed with elastic collisions with the sides. 

A \emph{generalized diagonal} for the polygon $P$ is a trajectory for the billiard flow that starts at one vertex of $P$ and ends at another vertex. Since the group $\Delta_P$ generated by reflections in the sides of $P$ is finite, the \emph{angle} of a trajectory is well defined in $S^1 \cong S^1/\Delta_P$. A motivating question for our paper is the following: \emph{how close in angle can two generalized diagonals of (less than) a given length be} (in terms of the length)?

Masur~\cite{Masur} showed that the number of generalized diagonals of length at most $R$ grows quadratically in $R$. We will show, for some families of billiards, that the smallest gap $\gamma^P_R$ between two generalized diagonals on $P$ of length at most $R$ satisfies \begin{equation}\label{eq:billiards:smallgap} \lim_{R \rightarrow \infty}R^2\gamma^P_R =0,\end{equation} and for other specific billiard tables that\begin{equation}\label{eq:billiards:nogap}\liminf_{R \rightarrow \infty} R^2 \gamma^P_R  >0.\end{equation}

\subsection{Organization of the paper}\label{subsec:organization} In the rest of \S\ref{sec:intro}, we describe the moduli space $\Omega_g$ of holomorphic differentials, and state the general versions of (\ref{eq:billiards:smallgap}) and (\ref{eq:billiards:nogap}). In \S\ref{sec:sl2:strata}, we recall the definition and properties of the $SL(2, \R)$-action on $\Omega_g$ and the decomposition of $\Omega_g$ into strata, and discuss the connection to billiards in \S\ref{subsubsec:billiards:translation}. We state our main technical results and important applications in \S\ref{sec:results}. We prove these results in \S\ref{sec:axiom}, following the approach in~\cite{EskinMasur}. In \S\ref{sec:measure:bounds}, we calculate bounds for measures of certain sets in $\hh$ and give examples of flat surfaces $\omega \in \hh$ with different types of gap behavior. We study a family of flat surfaces arising from rational billiards in \S\ref{sec:billiards}. 

\subsection{Translation surfaces}\label{subsec:holo} Let $\Sigma_g$ be a compact surface of genus $g \geq 2$.  Let $\Omega_g$ be the moduli space of holomorphic differentials on $\Sigma_g$. That is, a point $\om \in\Om_g$ is a equivalence class of pairs $(M, \omega)$, where $M$ is a genus $g$ Riemann surface, and $\omega$ is a holomorphic differential on $M$, i.e., a tensor with the form $f(z)dz$ in local coordinates, such that $\frac{i}{2}\int_{\Sigma_g} \omega \wedge \bar{\omega} = 1$. 

Two pairs $(M_1, \omega_1)$ and $(M_2, \omega_2)$ are equivalent if there is a biholomorphism $f:M_1 \rightarrow M_2$ such that $f_{*} \omega_1 = \omega_2$. For notational purposes, we will simply refer to the pair $(M, \om)$ by $\om$. Given $\om \in \Omega_g$, one obtains (via integration of the form) an atlas of charts on $\Sigma_g$ (away from the finite set of zeros of $\omega$) to $\C \cong \R^2$, with transition maps of the form $z \mapsto  z + c$. For this reason,  we will refer to $\omega \in \Omega_g$ as a \emph{translation surface} of genus $g$.

These charts determine a unique flat metric on $M$ with conical singularities at the zeros of the differential $\omega$. Geometrically, a zero of the form $z^k (dz)$ corresponds to a cone angle of order $(2k+2)\pi$. Zeroes of $\omega$ are \emph{singular} points for the flat metric. We refer to non-singular points as \emph{regular points}. The space $\Omega_g$ can be decomposed naturally into \emph{strata} $\hh$ (see \S\ref{sec:sl2:strata} for details), each carrying a natural measure $\mu_{\hh}$. 

\subsection{Saddle connections and cylinders}\label{subsec:sc} Fix $\om \in \Om_g$. A \emph{saddle connection} on $\om$ is a geodesic segment in the flat metric connecting two singular points (that is, zeros of $\om$) with no singularities in its interior. Given a regular point $p$, a \emph{regular closed geodesic} through $p$ is a closed geodesic not passing through any singular points. Regular closed geodesics appear in families of parallel geodesics of the  same length, which fill a cylindrical subset of the surface.

\subsubsection{Holonomy vectors}\label{subsubsec:holonomy} Let $\gamma$ be an (oriented) saddle connection or regular closed geodesic. Define the associated holonomy vector 
\begin{equation}\label{eq:sc:def} \vv_{\gamma} : = \int_{\gamma} \omega.\end{equation}

\noindent Note that if $\gamma$ is a closed geodesic, $\vv_{\gamma}$ only depends on the cylinder it is contained in, since regular closed geodesics appearing in a fixed cylinder all have the same length and direction. View $\vv_{\gamma}$ as an element of $\R^2$ by identifying $\C$ with $\R^2$. 

Let \begin{eqnarray}\label{eq:la:def}\La^{sc}_{\om} &=& \{ \mathbf{v}_{\gamma}: \gamma \mbox{ a saddle connection on } \omega\}\\ \La^{cyl}_{\om} &=& \{ \mathbf{v}_{\gamma}: \gamma \mbox{ a cylinder on } \omega\} \nonumber\end{eqnarray} 

\noindent be the set of holonomy vectors of saddle connections and cylinders respectively.  For $\La_{\om} = \La^{sc}_{\om}$ or $\La^{cyl}_{\om}$, we have that $\La_{\om}$ is discrete in $\R^2$ (see, e.g., ~\cite[Proposition 3.1]{Vorobets96}), but Masur~\cite{Masur:billiards} showed that associated set of directions
$$\Theta^{\om} : = \{\arg(\vv): \vv \in \Lambda_{\om}\}$$
\noindent is dense in $[0, 2\pi)$ for any $\omega \in \Omega_g$.

\subsection{Decay of gaps}\label{subsec:gap:decay}
In this paper, we give a measure of the quantitative nature of this density by considering fine questions about the \emph{distribution} of saddle connection directions. Given $R >0$, let 
\begin{equation}\label{eq:Theta:def} \Theta^{\om}_R  : = \{\arg(\vv): \vv \in \Lambda_{\om} \cap B(0, R)\}\end{equation} denote the set of directions of saddle connections (or cylinders) of length at most $R$. We write $\Theta^{\om}_R : = \{0 \le \theta_1 <\theta_2 < \ldots < \theta_n<2\pi\}$, where $n = \tilde{N}(\om, R)$   is the cardinality of $\Theta^{\om}_R$, and we view $\theta_{n+1}$ as $\theta_1$. Note that if we define $$N(\omega, R) : = |\La_{\om} \cap B(0, R)|,$$ we have $$\tilde{N}(\om, R) \le N(\om, R).$$ Let $\gamma^{\omega}(R)$ be the size of the smallest gap, that is $\gamma^{\omega}(R) = \min_{\theta_i \in \Theta_R} |\theta_i - \theta_{i+1}|$,   Masur~\cite{Masur} showed that the counting function $N(\omega, R)$ grows quadratically in $R$ for any $\omega$. Since there are at most finitely many (at most $4g-4$) saddle connections in a given direction, this shows that $\tilde{N}(\om, R)$ also has quadratic growth. Thus, one would expect the $\gamma^{\omega}(R)$ to decay quadratically. Our main theorem addresses the asymptotic behavior of the rescaled quantity $R^2 \gamma^{\omega}(R)$. Let $\hh$ be a stratum of $\Omega_g$, and let $\mu = \mu_{\hh}$.

\begin{Theorem}\label{theorem:main:gap} For $\mu$-almost every $\omega \in \hh$,
\begin{equation}\label{eq:main:gap}\lim_{R \rightarrow \infty} R^2 \gamma^{\omega}(R) = 0.
\end{equation}
\noindent Moreover, for any $\epsilon >0$, the proportion of gaps less than $\epsilon/R^2$ is positive. 
\begin{equation}\label{eq:main:proportion}
\lim_{R \rightarrow \infty} \frac{ |\{1 \le i \le \tilde{N}(\omega, R): (\theta_{i+1} - \theta_i) \le \epsilon/R^2\}|}{\tilde{N}(\omega, R)} >0.
\end{equation}
\end{Theorem}

\bigskip

\noindent Theorem~\ref{theorem:main:gap} cannot be extended to \emph{all} $\omega \in \hh$, since for any stratum $\hh$ there are many examples $\omega \in \hh$ for which 
\begin{equation}\label{eq:gap:lower}\liminf_{R \rightarrow \infty} R^2 \gamma^{\omega}(R) >0.\end{equation}

\noindent We say that $\omega$ has \emph{no small gaps} (NSG) if (\ref{eq:gap:lower}) holds. An important motivating example of a surface with NSG is the case of the square torus $(\C/\Z^2, dz)$. Since there are no singular points, there are no saddle connections, but cylinders are given by integer vectors, and $\Theta^{\om_0}$ then corresponds to rational slopes. It can be shown that $3/\pi^2$ is a lower bound for $R^2 \gamma^{\om_0}(R)$ (see, for example~\cite{BCZ1}). 

The torus is an example of a lattice surface. Recall that $\omega$ is said to be a \emph{lattice surface} if the group of derivatives of affine diffeomorphisms of $\omega$ is a lattice in $SL(2, \R)$ (see \S\ref{sec:sl2:strata} for more details). We have:

\begin{Theorem}\label{theorem:lattice:nsg} $\omega$ is a lattice surface if and only if it has no small gaps. \end{Theorem}
\medskip

We prove Theorem~\ref{theorem:lattice:nsg} in \S\ref{subsec:lattice}, using a result of Smillie-Weiss~\cite{nosmalltri} which characterizes lattice surfaces using the \emph{no small triangles} (NST) property defined by Vorobets~\cite{Vorobets96}. 

Theorem~\ref{theorem:main:gap} will follow from a precise statement about the asymptotic distribution of saddle connection directions. This generalizes work of Vorobets~\cite{Vorobets}, who showed that the sets $\Theta^{\omega}_R$ become uniformly distributed (as $R \rightarrow \infty$) in $[0, 2\pi)$ for almost every $\omega$ (in particular, for those with exact quadratic asymptotics of saddle connections). Our techniques are inspired by those of Marklof-Strombergsson~\cite{MS}, who studied the distribution of affine lattice points in Euclidean spaces by reducing them to equidistribution problems in homogeneous spaces. Much of the technical machinery is drawn from~\cite{EskinMasur}, in which Eskin-Masur give precise quadratic asymptotics of $N(\omega, \R)$ using equidstribution of translate of orbits under the $SL(2, \R)$-action on $\Omega_g$.

\subsubsection{Quadratic differentials}\label{subsubsec:qd} For notational convenience, we work with the space $\Omega_g$ instead of the space of quadratic differentials $Q_g$. A quadratic differential determines a flat metric, and so saddle connections and cylinders are well defined. For a saddle connection or cylinder curve $\gamma$, the holonomy vector $\vv_{\gamma}$ is given by integrating a square root of the differential, and are thus defined up to a choice of sign. The set of directions can then be viewed as a subset of $[0, \pi)$. Our results apply, \emph{mutatis mutandis}, except when explicitly indicated, to the setting of quadratic differentials.

\subsection{Hyperbolic angle gaps} Higher-genus translation surfaces can be viewed as an intermediate setting betwen flat tori and hyperbolic surfaces. Recently, Boca-Pasol-Popa-Zaharescu~\cite{BPPZ} study a spiritually similar problem in the hyperbolic setting. They calculate the limiting gap distribution for the angles (measured from the vertical geodesic) of hyperbolic geodesics connecting $i \in \h^2$ to points in its $SL(2, \Z)$-orbit. In this setting, they show the limiting distribution does have support at $0$, similar to the case of a generic translation surface.

\subsection{Acknowledgements}\label{subsec:ack} This paper was inspired by 
the beautiful paper~\cite{MS} on the distribution of  affine 
lattice points. We thank Alex Eskin, Jens Marklof and William Veech for 
useful discussions. Howard Masur not only patiently answered many technical questions about this paper but more generally taught us much of what we know about the subject. The initial discussions for the project took place while the authors were attending the Hausdorff Institute of Mathematics `Trimester Program on Geometry and Dynamics of Teichm\"uller Spaces' in Bonn. We would like to thank the Hausdorff Institute and the organizers of this program for their hospitality. The second author would like to thank the University of Illinois at Urbana-Champaign for its hospitality. The second author was supported in part by an NSF postdoc.
\section{The $SL(2, \R)$-action and strata}\label{sec:sl2:strata}

In this section, we describe the $SL(2, \R)$ action and stratification of $\Omega_g$ (\S\ref{subsec:polygons} - \S\ref{subsec:sl2:affine}), and the construction of an $SL(2, \R)$-invariant measure (\S\ref{subsec:coord:meas}). We also describe (\S\ref{subsubsec:billiards:translation}) the connection between rational billiards and translation surfaces. This is standard background material in the subject, and our exposition is brief, and drawing on~\cite{EMM, EMZ}. Excellent general references are~\cite{Zorich:survey, MasurTab}.

\subsection{Translation surfaces and polygons}\label{subsec:polygons}  A more geometric description of a translation surface can be given by a union of polygons $P_1 \cup \dots \cup P_n$ where each $P_i \subset \C$, and the $P_i$ are glued along parallel sides, such that each side is glued to exactly one other, and the total angle in each vertex is an integer multiple of $2 \pi$.  Since translations are holomorphic, and preserve $dz$, we obtain a complex structure and a holomorphic differential on the identified surface. The zeroes of the differential will be at the identified vertices with total angle greater than $2\pi$. The sum of the excess angles (that is, the orders of the zeros) will be $2g-2$, where $g$ is the genus of the identified surface.

\subsection{Combinatorics of flat surfaces}\label{subsec:comb:flat} The space $\Omega_g$ can be stratified by integer partitions of $2g-2$. If $\alpha =
(\alpha_1, \dots, \alpha_k)$ is a partition of $2g-2$, we denote by
$\hh(\alpha) \subset \Omega_g$ the moduli space of translation surfaces $(M,\omega)$
such that the multiplicities of the zeroes of $\omega$ are given by
$\alpha_1, \dots, \alpha_n$ (or equivalently such that the orders of
the conical singularities are $2 \pi (\alpha_1 + 1), \dots, 2
\pi(\alpha_n + 1)$). For technical reasons, the
singularities of $(M,\omega)$ should be labeled; thus, an element of
$\hh(\alpha)$ is a tuple $(M,\omega, p_1,\ldots,p_n)$, where
$p_1,\ldots,p_n$ are the singularities of~$M$, and the multiplicity
of~$p_i$ is~$\alpha_i$. The moduli space of translation surfaces is
naturally stratified by the spaces $\hh(\alpha)$; each is called a
{\em stratum}. Strata are not always connected, but Kontsevich-Zorich~\cite{KZ} (and Lanneau~\cite{Lanneau} in the setting of quadratic differentials) have classified the connected components.  Most strata are connected, and there are never more than three connected components.

\subsection{$SL(2, \R)$ and affine diffeomorphisms}\label{subsec:sl2:affine}

There is an action of $SL(2,\R)$ on the moduli space of
translation surfaces that preserves the stratification. Since $SL(2,\R)$ acts on $\C$ via linear maps on $\R^2$, given a surface $P_1 \cup \dots \cup P_n$, we can define $g S = g P_1 \cup \dots \cup
g P_n$, where all identifications between the sides of the polygons for
$gS$ are the same as for $S$. This action generalizes the action of
$SL(2,\R)$ on the space of (unit-area) flat tori $SL(2,\R)/SL(2,\Z)$. Note that $SL(2, \R)$ preserves the area of the surface $\omega$.

\subsubsection{Lattice surfaces}\label{subsubsec:lattice}
For $\omega \in \hh(\alpha)$, let $\Gamma(\omega) \subset SL(2,\R)$
denote the stabilizer of $\omega$. The group $\Gamma(\omega)$ is called the
{\em Veech group} of $S$. If $\Gamma(\omega)$ is a lattice in $SL(2,\R)$
then $\omega$ is called a {\em lattice surface}.

Equivalently, let $\mbox{Aff}(\omega)$ denote the set of affine (area-preserving) diffeomorphisms of $\om$. The derivative of any $f \in \mbox{Aff}(\om)$ will be a matrix in $SL(2, \R)$, and the collection $\{Df: f \in \mbox{Aff}(\om)\}$ coincides with $\Gamma(\omega)$ (up to a finite index subgroup). Thus $\Gamma(\omega)$ is a lattice if and only if $D(\mbox{Aff}(\omega))$ is.

\subsubsection{Billiards and translation surfaces}\label{subsubsec:billiards:translation} An important motivation for studying translation surfaces is their relationship to rational billiards. Recall that a polygon $P \subset \C$ is called rational if all angles of $P$
are rational multiples of $\pi$. The \emph{unfolding} procedure in ~\cite{Zelmjakov:Katok} describes how to associate a translation surface $\omega_P$ so that the billiard flow on $P$ is described by the geodesic flow on $\omega_P$. 

Let $\Delta_P \subset O(2)$ denote the group generated by reflections in the sides of the polygon $P$. Since $P$ is rational, $\Delta_P$ is finite. $\omega_P$ consists of $|\Delta_P|$ copies of $P$, with each copy glued to each of its mirror images along the reflecting side.

For example, if $P$ is the unit square, then $\omega_P \in \hh(\emptyset)$ is the torus
$\C/2 \Z \oplus 2 \Z$, and if $P$ is the $(\pi/8, 3\pi/8)$ right triangle, $\omega_P \in \hh(2)$ is a regular octagon with opposite sides identified.

\subsection{Coordinates and measure on strata}\label{subsec:coord:meas}

Let $\alpha = (\alpha_1, \ldots, \alpha_k)$ be an integer partition of $2g-2$. We describe how to put a topology and measure on $\hh(\alpha)$.  Our exposition is drawn from~\cite{EMZ}. For   a   flat   surface
$\omega_0 \in\hh(\alpha)$ with zero set $\Sigma = \{p_1,\dots, p_k\}$, choose a
basis  for     the     relative     homology
$H_1(\Sigma_g, \Sigma;\Z)$. We can pick a basis consisting of saddle connections, since we can choose saddle  connections that  cut $\om_0$ into a
union of polygons. For any $\omega$ near $\om_0$ holonomy vectors $\{\vv_{\gamma_i}\}$ yield local coordinates. That is, we view $\omega$ as an element of the relative cohomology
$H^1(\Sigma_g,\Sigma;\C) \cong \R^{4g+2k-2}$, and a domain in this vector space gives us a local  coordinate  chart.  We write $n = 4g+2k-2$. We normalize Lebesgue measure on $\R^n$ so that the integer lattice $\Z^n \cong H^1(\Sigma_g,\Sigma; \Z[i])$ has covolume $1$. Our measure $\mu(S)$ on $\hh(\alpha)$ is given by pulling back this measure via our coordinate maps. This is well-defined, the choice of volume element on $H^1(\Sigma_g,\Sigma;\C)$ is independent of choice of basis.

We will work with \emph{unit-area} surfaces.  Let  $\hh_1(\alpha)\subset\hh(\alpha)$ be the subset of unit area translation surfaces.  Let $a_i,b_i$,  $i=1,\ldots,g$ be a symplectic basis for homology $H_1(\Sigma_g, \Z)$. The area of the translation surface in the flat metric given by $\omega$ is given by

$$\int_{\Sigma_g} |\omega|^2 dx dy=\frac{i}{2}\int_{\Sigma_g}
\omega\wedge\bar\omega=
\frac{i}{2}\sum_i\left(\int_{A_i}\omega\int_{B_i}\bar\omega-
\int_{A_i}\bar\omega\int_{B_i} \omega\right).
$$
\noindent This can be viewed as an (indefinite) quadratic form in our local coordinates, and so the level set $\hh_1(\alpha)$ can be thought of as a `hyperboloid'. The measure on $\hh(\alpha)$ induces a measure on the hypersurface $\hh_1(\alpha)$. We can represent any
$\omega \in\hh(\alpha)$ as $\omega  = r \omega'$, where $r\in\R_+$, and $\omega' \in \hh_1(\alpha)$.
Holonomy vectors of saddle connections and cylinders  on $\omega'$ are
multiplied by $r$ to give  vectors  associated  to  corresponding
saddle  connections  on  $\omega$, and  $\mbox{area}(\omega) =
r^2\cdot\mbox{area}(\omega')=r^2$. The measure $\mu_1$ on
$\hh_1(\alpha)$ is given by disintegration of the volume element $\mu$
on $\hh(\alpha)$:
$$
d\mu(\omega) = r^{n-1} \, dr\, d\mu_1(\omega').
$$

\noindent In the sequel, by abuse of notation, we will fix a connected component $\hh$ of $\hh_1(\alpha)$ and denote the Lebesgue measure on it by $\mu$. We note that in any stratum, the set of surfaces arising from billiards as in \S\ref{subsubsec:billiards:translation} has measure zero. Thus, statements about almost every translation surface do not yield results about billiard flows, in particluar, Theorem~\ref{theorem:main:gap} does not apply to billiard flows. In \S\ref{sec:billiards} we discuss some special classes of billiards for which we can prove a version of Theorem~\ref{theorem:main:gap}.

\subsubsection{$SL(2,\R)$-invariance and ergodicity}\label{subsubsec:ergodic} Since the $SL(2,\R)$-action on $\hh(\alpha)$ preserves the area, it acts on the level set $\hh_1(\alpha)$. It also preserves connected components, so we can consider it acting on $\hh$. The measure $\mu$ constructed above is invariant under $SL(2,\R)$. The following theorem is due (independently) to Veech~\cite{Veech:gauss} and Masur~\cite{Masur:IET}.

\begin{theorem*}[Veech~\cite{Veech:gauss}, Masur~\cite{Masur:IET}] $\mu$ is a finite, ergodic, $SL(2,\R)$-invariant measure on $\hh$. \end{theorem*}

\subsubsection{Short saddle connections}\label{subsubsec:short:sc} We record here a crucial measure estimate on the set of surfaces with short saddle connections. It is originally due to Masur-Smillie~\cite{MasurSmillie} we recall it as it is quoted in~\cite[Lemma 7.1]{EMZ}

\begin{lemma}
[H.~Masur, J.~Smillie]
\label{lemma:short:saddle:connections}
There is a constant $M$ such that for all $\epsilon,\kappa>0$ the
subset of  $\hh$  consisting  of  those  flat surfaces,
which have a saddle connection of length at  most $\epsilon$, has
volume  at most  $M\epsilon^2$.  The volume of  the  set of  flat
surfaces with a saddle connection  of  length  at most $\epsilon$
and  a  nonhomologous  saddle  connection  with  length  at  most
$\kappa$ is at most $M\epsilon^2\kappa^2$.
\end{lemma}

\noindent\textbf{Remark:} Note that by the construction of the measure, the volumes of these sets will be \emph{at least} $m \epsilon^2$ and $m \epsilon^2 \kappa^2$ for some possibly smaller $m$, since we can construct local coordinates using a basis given by our short saddle connections.

\section{Saddle connections}\label{sec:results}

\noindent This section contains statements of our main results. In \S \ref{subsec:windows}, we give our main distribution result Theorem~\ref{theorem:wedge:ae} for the directions of saddle connections and cylinders. In \S\ref{subsec:proof:main:gap} we show how to use this result to derive Theorem~\ref{theorem:main:gap}. Theorem~\ref{theorem:wedge:ae} relies on results on limit measures for certain subsets of $SL(2, \R)$-orbits, which we describe in \S\ref{subsec:limit:circles}, and show how to use these limit theorems to obtain results for billiards. In \S\ref{subsec:lattice}, we describe how lattice surfaces yield exceptional behavior in our context. In \S\ref{subsec:limit:dist} we state results on measure bounds and gap distribution.

\subsection{Counting points in thinning segments}\label{subsec:windows} 

We recall notation: $\hh$ is a connected component of a stratum $\hh_1(\alpha)$ of unit area differentials in $\Omega_g$, and $\mu$ is Lebesgue measure on $\hh$, normalized to be a probability measure. Given $\omega \in \hh$, $\Lambda_{\omega}$ denotes set of holonomy vectors of either saddle connections or periodic cylinders in the flat metric determined by $\omega$. 

\subsubsection{The Siegel-Veech transform}\label{subsubsec:siegel:veech} We recall the defintion of the \emph{Siegel-Veech transform} from~\cite[\S2.1]{EskinMasur}. Given a compactly supported function $f: \R^2 \rightarrow \R$, define $\hat{f}: \hh \rightarrow \R$ by
\begin{equation}\label{eq:sv:transform}\hat{f}(\omega) = \sum_{\vv \in \La_{\om}} f(\vv).
\end{equation}

Veech~\cite{Veech} formulated the following seminal result, now known as the \emph{Siegel-Veech formula}:

\begin{Theorem}\label{theorem:sv:formula} Let $\eta$ be an ergodic $SL(2, \R)$ invariant probability measure on $\hh$. There is a $b = b(\eta)$ so that for all $f \in C^{\infty}_0(\R^2)$, $$\int_{\hh} \hat{f} d\eta = b \int_{\R^2} f dm$$
\end{Theorem}

\subsubsection{Thinning annular regions}\label{subsec:annular} Following \cite[\S 2.3]{MS}, we consider a family of thinning annular regions in $\R^2$ (see Figure~\ref{polar} below):  given $\theta \in [0, 2\pi)$, $\sigma, R >0$, and $0 \le c < 1$ define the annular region
\begin{equation}\label{eq:c:theta:sigma} 
A^{\theta}_R(c, \sigma) : = \{\vv \in \R^2: cR \le ||\vv|| \le R, \arg(\vv) \in (\theta - \sigma R^{-2}, \theta + \sigma R^{-2})\}.
\end{equation}
\makefig{The region $A = A^{\theta}_R(c, \sigma)$}{polar}{\input{wedge.pstex_t}}

\noindent As $R \rightarrow \infty$, this gives a narrowing wedge of directions around the angle $\theta$. For $\omega \in \hh$ define the counting function 

\begin{equation}\label{eq:n:theta:sigma} N^{\theta}_R(\omega, \sigma, c) : =  |\La_{\omega} \cap A^{\theta}_R(c, \sigma) |.\end{equation}

\noindent We think of $N^{\theta}_R(\omega, \sigma, c)$ as the number of saddle connections in a small neighborhood of the direction $\theta$. Note that $N^{\theta}_R(\omega, \sigma, c)$ can be viewed as the Siegel-Veech transform of the indicator function of $A^{\theta}_R(c, \sigma)$. Given that $|\La_{\omega} \cap B(0, R)|$ has quadratic asymptotics, one would expect $N^{\theta}_R(\omega, \sigma)$ to be proportional to $\sigma$. We frame the following question. Fixing an integer $k$, if we choose $\theta$ uniformly in $[0, 2\pi)$, what is the probability that $N^{\theta}_R(\omega, \sigma, c) = k$?

\begin{Theorem}\label{theorem:wedge:ae} Fix $\sigma >0$ and $c \in [0, 1)$, $k \in \Z_{\geq 0}$. Let $\hh$ be a (connected component of) stratum $\hh_1(\alpha)$ and let $\mu = \mu_{\hh}$ denote the natural $SL(2,\R)$-invariant probability measure on $\hh$. For $\mu$-a.e. $\omega_0 \in \hh$, 
\begin{equation}\label{eq:wedge:ae}
\lim_{R \rightarrow \infty} \lambda(\theta: N^{\theta}_R(\omega_0, \sigma, c) = k) = \mu(\omega: \La_{\om} \cap T(c, \sigma) = k)\end{equation}

\noindent where $\lambda$ denotes the Lebesgue probability measure on $[0, 2\pi)$ and $T(c, \sigma)$ is the trapezoid with vertices $(c, \pm c\sigma), (1, \pm \sigma)$.

\end{Theorem}

\medskip

\noindent In \S\ref{sec:axiom}, we will see that this theorem will hold with $\La_{\omega}$ replaced by other sets of holonomy vectors of special trajectories on $\omega$.

\subsection{Proof of Theorem~\ref{theorem:main:gap}}\label{subsec:proof:main:gap} We show how Theorem~\ref{theorem:main:gap} follows from Theorem~\ref{theorem:wedge:ae}. Given $k \in \Z_{\geq 0}, \sigma >0$, let
\begin{equation}\label{eq:pksigma:def} p_k(\sigma) := \mu(\omega: \La_{\om} \cap T(\sigma) = k)\end{equation}
\noindent where $T(\sigma) = T(0, \sigma)$. We require the following lemma.

\begin{lemma}\label{lemma:p2sigma} For any $\sigma>0$
\begin{equation}\label{eq:p2sigma} p_{2}(\sigma)>0. \end{equation}

\end{lemma}

\noindent\textbf{Proof:} As in the remark following Lemma~\ref{lemma:short:saddle:connections}, we note that we can construct in $\hh$ a set of measure at least $m \sigma^4$ ($m$ depending on $\hh$) with two saddle connections with holonomy vectors in $T(\sigma)$.
\qed
\medskip

\noindent Let $n \in \N$, let $\hh_{n}  \subset \hh$ be the full measure set of $\omega \in \hh$ so that (\ref{eq:wedge:ae}) holds for $\sigma = 1/n$, $c =0$. Then $\hh_{\infty} = \bigcap_{n=1}^{\infty} \hh_n$ is also a full measure set. We claim that for any $\omega_0 \in \hh_{\infty}$, (\ref{eq:main:gap}) holds, that is, 
$$\lim_{R \rightarrow \infty} R^2 \gamma^{\omega_0}(R) = 0.$$
\noindent Let $n \in \N$. Since $\omega_0 \in \hh_{\infty}$, (\ref{eq:wedge:ae}), there is an $R_n$ such that 
for all $R> R_n$, 
\begin{equation}\label{eq:lambda:positive}
\lambda(\theta: N^{\theta}_R(\omega_0, 1/n, 0) \geq 2) \geq p_2(\sigma)/2 >0. \end{equation}
\noindent The last inequality follows from Lemma~\ref{lemma:p2sigma}. By construction of $N^{\theta}_R(\omega_0, 1/n, 0)$, (\ref{eq:lambda:positive}) implies that for $R > R_n$, 
$$\gamma^{\omega_0}(R) \le \frac{1}{nR^2}.$$
\noindent Since $n$ was arbitrary (\ref{eq:main:gap}) follows. To see (\ref{eq:main:proportion}), note that $\lambda(\theta: N^{\theta}_R(\omega_0, \epsilon, 0) \geq 2)$ gives a lower bound for the limiting proportion of gaps of size less than $\epsilon/R^2$. Theorem \ref{theorem:main:gap} now follows from Theorem~\ref{theorem:wedge:ae} and Lemma~\ref{lemma:p2sigma}. \qed\medskip

\subsection{Limit measures}\label{subsec:limit:circles} The main weakness of Theorem~\ref{theorem:wedge:ae} is that it does not give us information about any particular surface $\omega_0 \in \hh$. To obtain such information, we must have further knowledge about the limiting behavior (as $t \rightarrow \infty$) of the orbits $\{g_t r_{\theta} \om_0: 0 \le \theta < 2\pi\}$, where 

 \begin{equation}
 \label{eq:matrices}
g_t = \left(\begin{array}{cc} e^{-t/2} & 0 \\ 0 & e^{t/2}
\end{array}\right),
r_{\theta} = \left(\begin{array}{cc} \cos\theta& -\sin\theta \\ \sin\theta & \cos\theta
\end{array}\right).
\end{equation}

\noindent Let $\nu_{t, \omega_0}$ denote the Lebesgue probability measure supported on $\{g_t r_{\theta} \om_0: 0 \le \om_0 < 2\pi\}$. Suppose $\lim_{t \rightarrow \infty} \nu_{t, \omega_0} = \mu_0$, and $\mu_0$ is $SL(2, \R)$-invariant. In this case we say $\mu_0$ is the \emph{circle limit measure} associated to $\omega_0$. By~\cite[Theorem 5.2]{EskinMasur}, $\mu_0$ is a probability measure. 

\begin{Theorem}\label{theorem:wedge:limit} Suppose $\omega_0 \in \hh$ has circle limit measure $\mu_0$. Then \begin{equation}\label{eq:wedge:limit}
\lim_{R \rightarrow \infty} \lambda(\theta: N^{\theta}_R(\omega_0, \sigma, c) = k) = \mu_0(\omega: \La_{\om} \cap T(c, \sigma) = k).\end{equation}
\end{Theorem}

\medskip 

\noindent Letting $\gamma^{0}(R) = \gamma^{\omega_0}(R)$, the proof of Theorem~\ref{theorem:main:gap} yields the following corollary to Theorem~\ref{theorem:wedge:limit}. Let $p^0_2(\sigma) : =  \mu_0(\omega: \La_{\om} \cap T(\sigma) = 2)$.

\begin{Cor}\label{cor:gap:limit} Fix notation as in Theorem~\ref{theorem:wedge:limit}. Suppose for all $\sigma >0$, $p^0_2(\sigma)>0$. Then 
\begin{equation}\label{eq:gap:limit}\lim_{R \rightarrow \infty} R^2 \gamma^{0}(R) = 0.
\end{equation}
\noindent Moreover, for any $\epsilon >0$, the proportion of gaps less than $\epsilon/R^2$ is positive. That is, writing $\Theta^{\om_0}_R : = \{0 \le \theta_1 \le \theta_2 \le \ldots \le \theta_n\}$, we have
\begin{equation}\label{eq:limit:proportion}
\lim_{R \rightarrow \infty} \frac{ |\{1 \le i \le \tilde{N}(\omega_0, R): (\theta_{i+1} - \theta_i) \le \epsilon/R^2\}|}{\tilde{N}(\omega_0, R)} >0.
\end{equation}
\end{Cor}

\medskip

\subsubsection{Billiards with barriers}\label{subsubsec:billiards:barriers} We describe how Theorem~\ref{theorem:wedge:limit} and Corollary~\ref{cor:gap:limit} can be used to give information about specific families of billiards. Following~\cite{EMS}, we consider the following family of billiards. Given $\alpha \in \R$, consider the polygon $P_{\alpha}$ whose boundary is the boundary of the square $[0, 1] \times [0, 1]$ together with a barrier given by the vertical segment $\{1/2\} \times [0, \alpha]$. 

\medskip

\noindent\textbf{Remark:} In fact, slits based at any rational point $p/q$ are considered in~\cite{EMS}. We restrict to $1/2$ for notational convenience and ease of exposition.

\medskip

The associated surface, which we denote by $\omega_{\alpha}$ is (after rescaling) an element of $\hh(1,1)$. A crucial observation is that $\omega_{\alpha}$ is a double cover of the torus. That is, there is a covering map (branched at the zeros) $\pi: \omega_{\alpha} \rightarrow \C/\Lambda$ where $\Lambda$ is a lattice in $\C$ and that $\omega_{\alpha}$ is obtained by pulling back the form $dz$. Using this construction, we observe that the set $\{\omega_{\alpha}: \alpha \in \R\}$ is contained in an $SL(2,\R)$-invariant subvariety $\M \subset \hh$ which can be identified with the moduli space of tori with marked points, $(SL(2,\R)\semidirect\R^2)/(SL(2,\Z)\semidirect\Z^2)$. 

Let $\mu_{\M}$ denote the natural $SL(2,\R)$-invariant probability measure supported on $\M$. Applying a theorem of Shah~\cite{Shah:SL2} (which uses Ratner's measure classification), it is shown in~\cite{EMS} that for any irrational $\alpha$, $\mu_{\M}$ is a circle limit measure for $\omega_{\alpha}$. We will see in \S\ref{subsec:translate} that for any $\sigma >0$,
$$p^{\M}_2(\sigma) := \mu_{\M}(\omega: \La_{\om} \cap T(\sigma) = 2)>0,$$

\noindent and thus, Corollary~\ref{cor:gap:limit} applies in this situation.

\subsection{Lattice surfaces}\label{subsec:lattice} In this section we prove Theorem~\ref{theorem:lattice:nsg}. Fix $\gamma^{\omega}_R$ to denote the smallest gap for saddle connections of length at most $R$. We split the proof into two lemmas. The first shows that any lattice surface has NSG.

\begin{lemma}\label{lemma:lattice:gap} For any lattice surface $\omega$ there exists a constant $\epsilon>0$ such that $\gamma^{\omega}_R\geq \frac{\epsilon}{R^2}$ for all $R>1.$
\end{lemma}
\begin{proof}We want to show that if $\omega \in \hh$ is a lattice surface, then $\omega$ has no small gaps. By \cite[Proposition 6.1]{Vorobets96} if $\omega$ is a lattice surface there exists a constant $s$ such that any two saddle connections in the same direction have the ratio of their lengths at most $s$. 
 Given a periodic direction $\theta$ there must be a cylinder of area at least $\frac 1 {4g-4}$ in that direction. Let $R$ be its length. By Lemma \ref{cyl flow} points in this cylinder are not in another cylinder of length at most $R$ in a direction within $\frac{1}{(8g-8)R^2}$. Any other saddle connection in this direction must have length $\frac{R} s$ which implies that a direction $\psi \in B(\theta,\frac{1}{(8g-8)R^2})$ can have no periodic cylinders of length less than $\frac R s$. \end{proof}
 
 \medskip
 
\noindent For the converse, we recall that $\omega$ has \emph{no small triangles} if there is a $\delta >0$ so that all triangles on $\omega$ with vertices at singularities, and no singularities in the interior have area at least $\delta$. Smillie-Weiss~\cite{nosmalltri} showed:

\begin{theorem*} $\omega$ is a lattice surface if and only if it has no small triangles. \end{theorem*}

\medskip

\noindent Combining this theorem with the following lemma completes the proof of Theorem~\ref{theorem:lattice:nsg}. 

\begin{lemma}\label{lemma:nst:nsg} If $\omega$ has no small gaps then it has no small triangles.\end{lemma}
\begin{proof} Let $\epsilon >0$ be such that $R^2 \gamma^{\om}_R > \epsilon$. Let $T$ be a triangle on $\omega$ with vertices at singularities and with no singularities in the interior. Without loss of generality we can assume that the sides of $T$ are saddle connections, since if not it can be decomposed into triangles which are. Let $R$ be the length of the longest side. Dropping a perpendicular from the opposite vertex, we decompose the side into two segments, at least one of which has length at least $R/2$. Consider the right triangle formed by the perpendicular and this segment. The angle opposite the perpendicular is an angle between saddle connections of length at most $R$, so it is at least $\frac{\epsilon}{R^2}$. See Figure~\ref{triangle} below.

 \makefig{$\theta \geq \frac{\epsilon}{R^2}$}{triangle}{\input{triangle.pstex_t}}

\noindent The length $L$ of the perpendicular is at least $\frac{R}{2}\tan(\frac{\epsilon}{R^2})$, so $L > \frac{R}{2}\frac{\delta}{R^2} = \frac{\delta}{2R}$ for some $\delta > 0$. Thus the area of the triangle is bounded below by $\frac{\delta}{4}$, and so is the area of $T$. Since $T$ was arbitrary, we have the $\omega$ has no small triangles.
\end{proof}

If we replace saddle connections with cylinders, then, as pointed out to us by Barak Weiss, we can construct a non-lattice surface with no small cylinder gaps as follows:  take a branched cover of a torus by varying relative periods (see \S\ref{sec:billiards}). This surface has absolute holonomy in
$\Z^2$ and therefore has no small gaps for vectors which are holonomies of
cylinder core curves. 
%

\subsection{Measure bounds and gaps}\label{subsec:limit:dist}
Let $\hat{G}_2(\Theta_R^{\omega})=\left\{R^2 (\theta_{i+1}-\theta_i)\right\}_{i=1}^{|\Theta_R^{\omega}|}$.
 Let $\nu_2^{\omega}(R)$ be the probability measure obtained by normalizing the measure given by delta mass at
 each element of $\hat{G}_2(\Theta_R^{\omega})$.
\begin{Prop}\label{lim distr} For almost every surface $\omega$ the measure $\nu_2^{\omega}(R)$ converges (as $R \rightarrow \infty$) in the weak-* topology.
\end{Prop}
\begin{proof} The existence of $p_0(\sigma)$ and the quadratic growth of saddle connections
implies this by \cite[Theorem 2.1]{MS}. In particular, after 
unwinding the definitions there we have 
$\nu_{\infty}([a,b])=\frac {d} {d\sigma} p_0(\sigma)|_{a}-\frac d {d\sigma}p_0(\sigma)|_b$. 
Note that this makes sense for almost every $\sigma$ because $p_0(\sigma)$ is a decreasing function of $\sigma$.
\end{proof}

\section{Equidistribution on strata}\label{sec:axiom}

In this section, we prove Theorem~\ref{theorem:wedge:ae} and and Theorem~\ref{theorem:wedge:limit}. We follow the strategy outlined in~\cite{EskinMasur} for proving results on the asymptotics of $N(\omega, R)$ and modify the techniques to our situation.

\subsection{Cones in $\R^2$}\label{subsec:cones}

The crucial geometric observation in the proof of Theorems~\ref{theorem:wedge:ae} and~\ref{theorem:wedge:limit} is the following. Let $R>>0$, and $t = 2 \log R$. Then 
$$r_{\theta} g_{-t} T(c, \sigma) \approx A^{\theta}_R(c, \sigma)$$
\noindent This can be seen as follows: $g_{-t} T(c, \sigma)$ is a trapezoid with vertices at $(cR, \pm c \sigma /R), (R, \pm \sigma/ R)$. Rotating it by angle $\theta$, we have that $r_{\theta} g_{-t} T(c, \sigma)$ is a thin trapezoidal wedge around the set line $\{\vv \in \R^2: \arg(\vv) = \theta\}$, with vectors of length roughly  between $cR$ and $R$ (there is an error of up to $ \frac 1 R$ because $r_{\theta}g_{-t}T(c,\sigma)$ is a trapezoid and not a wedge of an annulus). Finally, the angular width of this wedge is roughly $c/R^2$, since for small angles $\phi$ the slope $\tan(\phi) \approx \phi$. Thus, for any $\omega \in \hh$, $R >>0$,
\begin{equation}\label{eq:lambda:wedge}|\Lambda_{\om} \cap A^{\theta}_R(c, \sigma)| \approx |g_{t} r_{-\theta} \Lambda_{\om} \cap T(c, \sigma)|.\end{equation}

\noindent For $k \in \N$, let $f_k: \hh \rightarrow [0, 1]$ be the indicator function of the set 
\begin{equation}\label{eq:level:set}\hh_{c, \sigma, k} : = \{\omega \in \hh: |\Lambda_{\om} \cap T(c, \sigma)| = k\}.\end{equation}

\noindent Using (\ref{eq:lambda:wedge}), we can write
\begin{equation}\label{eq:circle:relation}
 \lambda(\theta: N^{\theta}_R(\omega, \sigma, c) = k) \approx \int_0^{2\pi} f_k(g_t r_{-\theta} \omega) d\lambda(\theta) =  \int_{\hh} f_k d\nu_{t,\omega}.
\end{equation}
\noindent Here, $\approx$ denotes that the difference goes to $0$ as $R \rightarrow \infty$. We will rigorously justify (\ref{eq:circle:relation}) below.
\subsection{Equidistribution}\label{subsec:equi} Equation (\ref{eq:circle:relation}) reduces proving Theorem~\ref{theorem:wedge:ae} and Theorem~\ref{theorem:wedge:limit} to understanding 
$$\lim_{t \rightarrow \infty} \int_{\hh} f_k d\nu_{t,\om}.$$We first prove Theorem~\ref{theorem:wedge:limit} assuming (\ref{eq:circle:relation}):

\noindent\textbf{Proof of Theorem~\ref{theorem:wedge:limit}:} While there is no general pointwise ergodic theorem known for the measures $\nu_{t,\om}$, our functions $f_k$ are indicator functions of level sets of the Siegel-Veech transform of the indicator function $h$ of $T(c, \sigma)$ (see \S\ref{subsubsec:SV} below). Following~\cite[\S4]{EskinMasur}, there is an approximation argument that allows us to conclude that if $\nu_{t,\om_0} \rightarrow \mu_0$, we have, as desired
$$\lim_{t \rightarrow \infty} \int_{\hh} f_k d\nu_{t, \om_0} =  \int_{\hh} f_k d\mu_0 = \mu_0(\om \in \hh: \La_{\om} \cap T(c, \sigma) = k).$$
\noindent\qed\medskip

\noindent\textbf{Proof of Theorem~\ref{theorem:wedge:ae}:} Combine Theorem~\ref{theorem:wedge:limit} with~\cite[Proposition 3.3]{EskinMasur}.
\qed\medskip

\subsubsection{Proof of (\ref{eq:circle:relation})}\label{sec:circle:relation}

Let $F_t$ denote the indicator function of $g_{t} r_{-\theta} A^{\theta}_R(c, \sigma)$. By construction, this does \emph{not} depend on $\theta$, as it is the $g_{t}$ image of a horizontal cone. As above, let $h$ denote the indicator function of $T(c, \sigma)$. We abbreviate $\nu_{t, \om_0}$ by $\nu_t$. Denote  the symmetric difference of $g_t r_{-\theta} A^{\theta}_R(c, \sigma)$ and $T(c, \sigma)$ by $E_t$, and note that $E_s \subset E_t$ for $s > t$, and the volume of $E_t$ tends to zero uniformly in $\theta$ as $t \rightarrow \infty$. Also note that by construction, $$\widehat{F_t}(g_{t} r_{-\theta} \omega) = N^{\theta}_R(\omega, \sigma, c),$$ and $f_k$ is the indicator function of the level set $\{\omega: \widehat{h}(\omega) = k\}$.  Thus we would like to show that, given any basepoint $\omega_0$, $$\lim_{t \rightarrow \infty} \nu_t (\omega: \widehat{F_t} (\omega) = k) = \lim_{t \rightarrow \infty} \nu_t(\omega: \widehat{h}(\omega) = k).$$ Recall that for any set $L \subset \hh$, $$\nu_t(\omega: \omega \in L) = \lambda(\theta: g_{-t} r_{-\theta} \omega_0 \in L).$$

\noindent Let $I_t = \{\om: \widehat{F_t}(\omega) \neq \widehat{g}(\om)\}$. Letting $b_0$ denoting the Siegel-Veech constant (see Theorem~\ref{theorem:sv:formula}) of $\mu_0$, we have $$\mu_0(I_t) < b_0 m(E_t),$$where $m$ denotes Lebesgue measure on $\R^2$. Also note that $I_s \subset I_t$ for $s > t$. Fix $\epsilon >0$. Let $t_0 >0$ be such that $m(E_t) < \epsilon/C_0$ for all $t > t_0, \theta \in [0, 2\pi)$, so $$\mu_0(I_t) < \epsilon.$$

\noindent Since $\nu_t \rightarrow \mu_0$, we can pick $t_1 >t_0$ so that for all $t > t_1$, $$\nu_t (I_t) \le \nu_t(I_{t_0}) \le 2 \epsilon.$$

\noindent Write $$\nu_t(\omega: \widehat{F_t}(\om) = k ) = \nu_t(\omega \notin I_t: \widehat{F_t} = k) + \nu_t(\omega \in I_t: \widehat{F_t}(\om) = k)$$

\noindent The second term can be bounded above by $\nu_t(I_t)$ and thus by $2 \epsilon$. The first term can be rewritten as 
$$\nu_t(\omega: \widehat{g} (\om) = k) - \nu_t(\om \in I_t: \widehat{g}(\om) = k)$$

\noindent Again using our bound on $\nu_t(I_t)$,  the difference $$|\nu_t(\omega: \widehat{F_t}(\om) = k ) -\nu_t(\omega: \widehat{g} (\om) = k)|$$ can be bounded by $4 \epsilon$. Passing to the limit, we obtain that the limits can differ by no more than $4 \epsilon$. Since $\epsilon$ was arbitrary, we have that the limits must be equal.  \qed\medskip

\subsubsection{The Siegel-Veech formula}\label{subsubsec:SV} Using the Siegel-Veech transform, we can obtain results on the \emph{expected} number of points in a thinning wedge. We fix some notation. Recall that for a bounded compactly supported function $f: \R^2 \backslash \{0\} \rightarrow \R$, we define $\widehat{f}(\om) = \sum_{\vv \in \La_{\om}} f(\vv)$. Fixing $R, c, \sigma$, let $h_{\theta} = \chi_{A^{\theta}_R(c, \sigma)}$, that is, it is the indicator function of the thinning wedge. The expected number of lattice points in a random thinning wedge can be written as

$$\sum_{k=1}^{\infty} k  \lambda(\theta: N^{\theta}_R(\om, \sigma, c) = k) = \int \widehat{h_{\theta}}d\lambda(\theta).$$

\noindent Writing $h$ for the indicator function of $T(c, \sigma)$, and using (\ref{eq:lambda:wedge}) we can write this (for $t>>0$) as 
$$\int_{\hh} \widehat{h} d\nu_{t,\omega}.$$

\noindent If $\omega_0$ has a circle limit measure $\mu_{0}$ satisfying a certain technical condition~\cite[Theorem 8.2(D)]{EMM}, then 
$$\lim_{t \rightarrow \infty} \int_{\hh} \widehat{h} d\nu_{t,\omega} = \int_{\hh} \widehat{h} d\mu_0.$$

\noindent As above, let $b_0$ be the Siegel-Veech constant for $\mu_0$. Then we have
\begin{equation}\label{eq:siegel:veech}\int_{\hh} \widehat{h} d\mu_0 = b_0 \int_{R^2} h.\end{equation}

\noindent That is, the expected number of points in a thinning wedge is proportional to the volume of the trapezoid $T(c, \sigma)$. We will see in \S\ref{subsec:second:moment} that the \emph{second moment} of the limiting distribution is \emph{not} well-defined.

\subsubsection{Fiber bundles and special trajectories}\label{subsubsec:fiber} In this paper, we focus on the sets of holonomy vectors of oriented saddle connections or cylinders. Our results will also apply to the sets of holonomy vectors connecting a fixed point on the surface to singular points or the set of vectors connecting two marked points. These can be obtained by considering the corresponding equidistribution results for the spaces $Y_i$ of translations surfaces with $i$ marked points, $i=1, 2$. These spaces can be viewed as fiber bundles over $\hh$ with fiber over $\omega \in \hh$ given by $(M, \omega)^i$. We refer the interested reader to~\cite[\S2, \S9]{EskinMasur} for details.

\section{Measure bounds and gap distribution}\label{sec:measure:bounds}
This section provides a variety of results on the gaps between saddle connection and cylinder directions. It includes results on the likelihood of finding many saddle connections in a small region. In particular,  Lemma \ref{k asym} says that having $k$ saddle connections in a small interval is proportional to having two saddle connections in a small interval. As a corollaries we show Corollary \ref{cor:second:moment} which says that the limiting distribution of $p_k(\sigma)$ does not have finite second moments and Corollary \ref{SV not L2} which says that the Siegel-Veech transform does not send continuous compactly supported functions to $L_2$ (though it is norm preserving for positive functions as a map from $L_1$ to $L_1$). We also show Theorem \ref{p sig 0 lower} which says that some gaps between cylinder directions are larger than one would expect .
\subsection{Notation}\label{subsec:not}

In this section, we will need to distinguish between holonomy vectors of saddle connections and cylinders. Let $\omega \in \hh$, let $\La^{sc}_{\om}$ and $\La^{cyl}_{\om}$ be as in (\ref{eq:la:def}), and let $\mu$ be Lebesgue measure on $\hh$. We define

\begin{eqnarray}\label{eq:tildepksigma:def} p_k(\sigma)  :&=& \mu(\omega: \La^{sc}_{\om} \cap T(\sigma) = k) \\ \nonumber \tilde{p}_k(\sigma) :&=& \mu(\omega: \La^{sc}_{\om} \cap T(\sigma) \geq k)\\ \nonumber p^{cyl}_k(\sigma)  :&=& \mu(\omega: \La^{cyl}_{\om} \cap T(\sigma) = k) \\ \nonumber \tilde{p}^{cyl}_k(\sigma) :&=& \mu(\omega: \La^{cyl}_{\om} \cap T(\sigma) \geq k).\\ \nonumber \end{eqnarray}

\subsection{$\sigma \rightarrow 0$ asymptotics}\label{subsec:sigma:0}
\begin{Prop}$ \underset{\sigma \to 0}{\lim}\, p_{\sigma,1}=0$. 
\end{Prop}

\begin{proof} This follows from the fact that the volume of $T(\sigma) \rightarrow 0$ as $\sigma \rightarrow 0$.\end{proof}

\begin{Theorem}\label{sig to 0} As $\sigma \rightarrow 0$,
$\sigma^{-2} p_{\sigma,2}$ is bounded away from 0 and infinity.
\end{Theorem}
\begin{proof}
Consider $g_{-\sqrt{\sigma}} T(c,\sigma)$. Because the action of
$g_t$ preserves $\mu$ on $\mathcal{H}$ it follows that the measure
of surfaces with two saddle connections $T(c, \sigma)$ is equal to
the measure of surfaces with two saddle connections in
$g_{-\sqrt{\sigma}}T(c,\sigma)$. It follows from  Lemma
\ref{lemma:short:saddle:connections} and the remark following it that $$\mu(\{\omega:
\Lambda_{\omega}\cap T(c,\sigma)\geq 2\}) \sim \sqrt{\sigma}^2\sqrt{\sigma}^2=\sigma^2,$$ where $\sim$ denotes proportionality.
\end{proof}

This result states that given a saddle connection appearing in a wedge the probability of having another one is roughly independent. In general this is false, as the next section shows the probability of having many saddle connections in a small wedge decays slowly.

\subsection{More on $\sigma \rightarrow 0, k \rightarrow \infty$ }
First,
\begin{lemma} \label{2k above} For any fixed $\sigma>0$ we have $\underset{k \to \infty}{\limsup} \, k^2 \tilde{p}^{\text{cyl}}_k(\sigma)<\infty$.
\end{lemma}

This is similar to the proof of Theorem \ref{sig to 0}. We say a surface has an $\epsilon$\emph{-thin neck} if there exists a pair of saddle connections $\vv$
 and $\ww$ such that $|\vv| \leq \epsilon$ and $\ww$, $\vv$, $\ww$
  are adjacent saddle connections. See Figure~\ref{thin neck} below.   
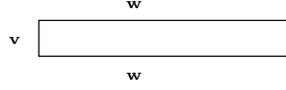
  \makefig{An $\epsilon$-thin neck, i.e., $|\vv| < \epsilon$.}{thin neck}{\input{thin_neck.pstex_t}}

  If $\vv$, $\ww$ define a $\frac{\sigma}{3k}$ thin neck in $\omega$ and 
  $\frac 1 3 \leq |\ww| \leq \frac 2 3$ then $N_1^{\arg (\ww)}(\omega,\sigma,0)\geq k$.
   To see this notice that a saddle connection that goes once in the $\ww$ direction and $l$ times in the $\vv$
    direction has its associated vector in $\Lambda_{\omega}$ contained in  $A_1^{\arg(\ww)}(0,\sigma)$. 
\begin{lemma}\label{k asym} For any fixed $\sigma >0$ we have $ k^2 \tilde{p}_k(\sigma)$ is bounded away from 0 and $\infty$.
\end{lemma}
\begin{lemma} \label{sig asym} For any fixed $k>1$ we have $\sigma^{-2} \tilde{p}_k(\sigma)$ is bounded away from $0$ and $\infty$.
\end{lemma}
These two lemmas follow from the previous paragraph by noticing that the measure of surfaces with an $\epsilon$-thin neck is proportional to $\epsilon^2$ (see \S \ref{subsec:coord:meas}). Showing that it is bounded away from $\infty$ follows from Theorem \ref{sig to 0} and Lemma \ref{2k above}.

\subsubsection{Non-existence of second moments}\label{subsec:second:moment} Lemma~\ref{k asym} has the following corollary.

\begin{Cor}\label{cor:second:moment} For any $\sigma >0$, $$\sum_{k =0}^{\infty} k^2 p_k(\sigma) \mbox{ diverges }.$$ That is, the limiting distribution $p_k(\sigma)$ does not have finite second moment.\end{Cor}

\medskip
\noindent The corollary follows from Lemma~\ref{k asym} and the general lemma below:

\begin{lemma}\label{lemma:prob} Let $X$ be a positive integer valued random variable, with $P(X = k) = p_k$. Suppose there is a $c >0$ so that $P(X \geq k ) \geq \frac{c}{k^2}$. Then
$$\sum_{k =0}^{\infty} k^2 p_k \mbox{ diverges}.$$
\end{lemma}

\begin{proof} We are interested in calculating $E(X^2) = \sum_{k=0}^{\infty} k^2 p_k$. We can write
$$E(X^2) = \sum_{k=0}^{\infty} P(X^2 \geq k)$$ For $i^2 \le k < (i+1)^2$, $P(X^2 \geq k) = P(X \geq i)$. Thus $$\sum_{k=0}^{\infty} P(X^2 \geq k) = \sum_{i=1}^{\infty} 2i P(X \geq i) \geq \sum_{i=1}^{\infty} \frac{c}{i}.$$\end{proof}


\subsubsection{Glued-in tori} For the remaining estimates we consider gluing in a small torus.
 We say a surface has a \emph{glued in
torus} with parameters $(a,b)$ and gluing slit $s$ there is a portion of the surface where the points travel as if they are in a torus with basis lengths $a,b$ except if they cross a saddle connection of length at most $s$.  See Figure~\ref{almost torus} below.

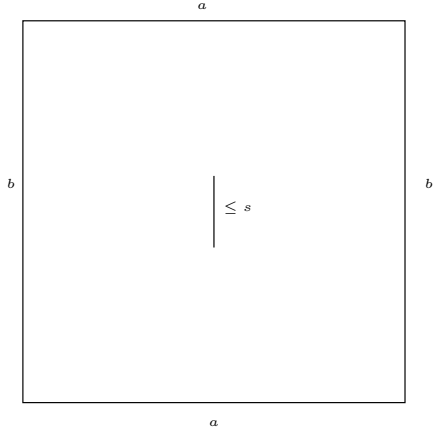
  \makefig{The slit torus is glued to the rest of the surface along the saddle connection of length at most $s$.}{almost torus}{\input{almost_torus.pstex_t}} 

\begin{Prop} 
 The measure of unit volume surfaces that have a glued in torus with
parameters $(a,b)$ and gluing slit $s$ where $a, b \in
[\sqrt{c},2\sqrt{c}]$ and $s \in [c,2c]$ is at least proportional to $c^{-4}$
as $c $ goes to zero.
\end{Prop}
This follows from the main result in \S \ref{subsec:coord:meas}. If we glue in a torus with parameters comparable to $(\sqrt{\sigma}$,
$\sqrt{\sigma})$ with gluing slit $\sigma$ it has quadratic growth of
saddle connections or periodic cylinders. Because a torus is a
lattice surface, the saddle connection directions are completely
periodic. If the length of the periodic cylinder is less than $t$
then the trajectory in the torus crosses the direction of the slit
at most $c\frac{t}{\sqrt{\sigma}}=ct\sqrt{\sigma}^{-1}$ times.  
So some points do not leave the torus before
closing up. Therefore all saddle connection directions of the torus
with length less than $\frac{\sigma}{c}$ are cylinder directions for the
surface.

It follows that there exists $C>0$ such that for any $\sigma>0$ small
enough a surface that has a torus with gluing parameters
$\sqrt{\sigma}$, $\sqrt{\sigma}$, not too small angle between these sides and gluing slit $\sigma$ has at
least $C \sigma^{-1}$ periodic directions whose cylinders have
length less than or equal to 1. It follows from the pigeonhole
principle that almost half of these directions are separated by at
most $4\pi C \sigma^{-1}$.


\begin{lemma}
For any fixed $k>0$ we have
$$\underset{\sigma \to 0}{\liminf}\,
\sigma^{-4}\, \tilde{p}^{cyl}_k(\sigma)>0.$$
\end{lemma}

\begin{proof}Glue in a torus of parameters comparable to $\sqrt{\sigma}$,$\sqrt{\sigma}$ with gluing slit $\sigma$.\end{proof}
\begin{lemma}For any fixed $\sigma>0$ we have
$$\underset{k \to \infty}{\liminf}\, k^{4} \tilde{p}^{cyl}_k>0.$$
\end{lemma}
\begin{proof}Glue in a torus of parameters comparable to $\sqrt{2 \sigma
k}^{-1}$, $\sqrt{2\sigma k}^{-1}$ with gluing slit $(2\sigma k)^{-1}$.\end{proof}

\begin{Cor}\label{SV not L2} Let  $f: \mathbb{R}^2 \backslash \{0\} \to \mathbb{R}$ by $f(x)=1$ if $x \in B(0,1) \backslash \{0\}$ and 0 otherwise. Its Siegel-Veech transform $\hat{f}$ is not in $L^{2}(\hh, \mu)$. However $\hat{f}$ is in $L^{2-\epsilon}(\hh,\mu)$ for any $\epsilon>0$.
\end{Cor}
\medskip

\noindent This is an immediate consequence of Lemma \ref{k asym}. Notice that $f$ is in $L^{\infty}(\mathbb{R}^2)$. The above corollary fails if we take the analogue of the Siegel-Veech transform for the directions of cylinders. That is, let 
$$\tilde{f}(\omega)= \sum_{\vv \in \La_{\om}^{cyl}} f(\vv) .$$  In this case the function is not in $L^{4}(\hh, \mu)$. However, if we fix the minimal volume of cylinders we consider then by Lemma \ref{cyl flow} this variant of the Siegel-Veech transform sends $L^{\infty}(\mathbb{R}^2\backslash \{0\})$ to $L^{\infty}(\hh,\mu)$. The $L^{\infty}$ norm may increase and that this increase can be bounded by the genus  of the surfaces parametrized by $\hh$ and the lower bound on the volume of the cylinders. Thus while the different versions of the Siegel-Veech transform are all $L^1$ norm preserving on positive functions they have different behavior in $L^p$ in general.

\subsection{$\sigma \rightarrow \infty$ asymptotics}\label{subsec:sigma:infty}
\begin{lemma}\label{cyl flow} Suppose $x$ is in a periodic cylinder of length $L$, area $a$ in direction $\theta$. Then $x$ is not in a  periodic cylinder of length less than $R$ and direction $\theta'$ where $ \theta' \neq \theta$ and $|\theta'-\theta|<\min\{\frac 1 {2LRa}, \frac{\pi}{3}\}$.
\end{lemma}
\begin{proof} Consider the periodic cylinder in the hypothesis of the lemma.
 It has a width vector $\vv$ where $|\vv|=\frac{a}{L}$ and let $x \in \vv$. If $x+L\tan(\epsilon)< \frac a L$ then $F_{\theta+\epsilon}|_{\vv}(x)=x+ L \tan(\epsilon)$, where $F_{\theta+\epsilon}|_{\vv}$ denotes the induced map of  $F_{\theta+\epsilon}$ on $\vv$. It follows that if $\theta+\epsilon$ is the direction of a periodic cylinder $x$ lies in we have $F_{\theta+\epsilon}|_{\vv}^k(x)=x$ and so  $kL\tan(\epsilon)\geq |\vv|.$ The length of this cylinder is at least $kL\sec(\epsilon)$. Therefore if $kL<R$ then $\tan(\epsilon)>\frac{|\vv|}{R}=\frac{a}{RL}.$ Noticing that $\epsilon< \frac{\pi}{3}$ implies $\tan(\epsilon)<2\epsilon$ completes the lemma.
\end{proof}

 
 \subsubsection{Absence of cylinder gaps}\label{subsec:cylinder:gap}
Let $\Theta_{\omega}^{*}(R)$ be the directions $\theta$ such that the volume of the periodic cylinders in direction $\theta$ with length than or equal to $R$ is at least $\frac{3}{5}$. In the rest
\begin{lemma}\label{farey} Let $\theta \in \Theta_{\omega}^{*}(L)$ then
 $$\left(\theta - \frac{1}{20(4g-4)},\theta+ \frac{1}{20(4g-4)RL}\right)
\cap \Theta_{\omega}^{*}(R)=\theta.$$
\end{lemma}
\begin{proof} One can easily see that the number of 
saddle connections in a given direction is at most $4g-4$.
 Therefore $\theta \in \Theta_{\omega}^{*}(L)$ 
then at least half the points in the surface must lie in 
cylinders of area at least $\frac{1}{10(4g-4)}$.
 If $\phi \in \Theta_{\omega}^{\frac3 5}(R)$ then some points of the surface
 must lie in periodic cylinders in direction $\phi$ and
 periodic cylinders of area at least $\frac 1 {10(4g-4)}$
 in direction $\theta.$
The result follows from Lemma \ref{cyl flow}.
\end{proof}
In the square torus, periodic directions correspond to rational slopes. Notice that if $\frac a L \neq  \frac b R $ are rational numbers and $CL<R$ then $|\frac a L- \frac b R|>\frac{C}{R^2}$. This implies that for any $C'>0$ a positive proportion of periodic directions of length less than $R$ are at least $\frac{C'}{R^2}$ separated from the closest periodic direction of length less than $R$. The next theorem generalizes this fact for periodic cylinders of substantial area.

\begin{Theorem}\label{p sig 0 lower}
If the cardinality of $\Theta_{\omega}^{*}(R)=\{\theta_1 \leq \theta_2 \leq \dots \theta_N\}$
 grows quadratically then 
 $$\underset{\sigma \to \infty}{\liminf} \, \underset{R \to \infty}{\liminf} \,\sigma^2 \frac{|\{\theta_i \in \Theta_{\omega}^{*}(R): \theta_{i+1}-\theta_i \geq \frac{\sigma}{R}\}|}{|\Theta_{\omega}^{\frac 3 5 }(R)|}>0.$$ 
\end{Theorem}
\begin{proof} By the fact that $\omega$ has quadratic growth there exists $c>0$ such that $$\frac{|\Theta_{\omega}^{*}(L)|}{|\Theta_{\omega}^{*}(R)|}\geq c \left(\frac L R\right)^2.$$  Lemma \ref{farey} implies that if
 $\theta \in \Theta_{\omega}^{*}(L)$ then 
it has a gap of size at least $\frac 1 {20(4g-4)RL}$ to the closest
direction in $\Theta_{\omega}^{*}(R)$. Let $L=\frac {R}{\sigma 20(4g-4)}$.
\end{proof}

\section{Billiards}\label{sec:billiards} In this section we use Theorem~\ref{theorem:wedge:limit} to obtain results about special trajectories for billiards, as discussed in \S\ref{subsubsec:billiards:barriers}. We follow closely the exposition in~\cite{EMM, EMS}, particularly focusing on the examples studied in the latter paper.

\subsection{Rectangles with barriers}\label{subsec:barrier} We recall notation from \S\ref{subsubsec:billiards:barriers}: given $\alpha \in \R$, consider the polygon $P_{\alpha}$ whose boundary is the boundary of the square $[0, 1] \times [0, 1]$ together with a barrier given by the vertical segment $\{1/2\} \times [0, \alpha]$. 

We recall the `unfolding' procedure from~\cite{EMS}: to obtain a translation surface $\omega_{\alpha}$ from $P_{\alpha}$, take four copies of $P= P_{\alpha}$ which are images of $P$
under reflection in the two coordinate axes and reflection in the origin. 

Identifying the interior sides, we obtain a square of area 4 with two vertical double lines, corresponding to the interval $1/2 \times [0,\alpha]$. Identify the top and bottom of the square, and the left and right sides. The glue the left side of the
right line to the right side of the left line, and the right side of
the right line to the left side of the left line. See Figure~\ref{unfold} below.

\makefig{Unfolding $P_{\alpha}$ to $\omega_{\alpha}$.}{unfold}{\input{developwall.pstex_t}}

Rescaling by $1/4$, we obtain an area $1$ translation surface $\omega_{\alpha} \in \hh(1,1)$, with the $2$ zeroes located at the endpoints of the vertical lines. A billiard trajectory $\lambda$ on $P_{\alpha}$ corresponds to a straight line on $\om_{\alpha}$.

\subsection{Branched covers}\label{subsec:branch} As mentioned in \S\ref{subsubsec:billiards:barriers}, the crucial property of the surface $\om_{\alpha}$ is that it is a double (branched) cover of the torus. That is, there is a covering map (branched at the zeros) $\pi: \omega_{\alpha} \rightarrow \C/\Lambda$ where $\Lambda$ is a lattice in $\C$ and that $\omega_{\alpha}$ is obtained by pulling back the form $dz$. 

The set of all $\om \in \hh(1, 1)$ satisfying this property is a closed, $SL(2, \R)$-subvariety $\M$~\cite[Lemma 2.1]{EMS}. $\M$ is a finite cover of $\T = (SL(2, \R) \semidirect \R^2)/(SL(2, \Z) \semidirect \Z^2)$ and the covering map commutes with the $SL(2,\R)$-action~\cite[Lemma 2.2]{EMS}. Recall that  $\T$ is the space of tori with two marked points (assuming that one marked point is always at the origin). The covering map $\Pi: \M \rightarrow \T$ is given by 
$$\Pi(\omega) = (\Lambda, \pi(z_1), \pi(z_2)),$$ 
\noindent where $\Lambda \subset \C$ is the lattice so that $\omega$ covers $\C/\Lambda$ and $z_1, z_2$ are the zeros of $\omega$. 

Let $\alpha \in \R$ be irrational. For $\omega_{\alpha} \in \M$, let $\M(\alpha)$ denote the connected component of $\M$ containing $\om_{\alpha}$. Let $\bar{\mu}$ denote the pullback of the Haar probability measure on $\T$. It is an ergodic $SL(2, \R)$ invariant measure on $\M(\alpha)$. The circle limit measure for $\omega_{\alpha}$ is $\bar{\mu}$~\cite[Lemma 2.4]{EMS}. Thus, we obtain the following corollary to Theorem~\ref{theorem:wedge:limit}.

\begin{Cor}\label{cor:wedge:billiard} For any irrational $\alpha$,
$$\lim_{R \rightarrow \infty} \lambda(\theta: N^{\theta}_R(\omega_{\alpha}, \sigma, c) = k) = \bar{\mu}(\omega \in \M(\alpha): \La^{sc}_{\om} \cap T(c, \sigma) = k).$$
\end{Cor}

\subsubsection{Branched covers of lattice surfaces}\label{subsubsec:branch:lattice} A similar calculation of circle limit measures was carried out for branched covers of lattice surfaces in~\cite{EMM}. Thus, we can obtain a version of Corollary~\ref{cor:wedge:billiard} for these surfaces as well. This will yield results for triangular billiards $P_n$ with angles
  $$ \frac{n-2}{2n} \pi, \ \frac{n-2}{2n} \pi, \ \frac{4}{2n} \pi
  ,$$ where $n \ge 5$, $n$ odd. For details, see~\cite[\S9]{EMM}. 

\subsection{Lattice translates}\label{subsec:translate} To obtain a version of Theorem~\ref{theorem:main:gap} for $\om \in \M$, we have to understand the set of holonomy vectors of saddle connections $\La^{sc}_{\om}$. Suppose $\om$ is a cover of $\C/\La$, with covering map $\pi$. Note that under the projection to the torus, the saddle connections must connect either $\pi(z_1) = 0$ or $\pi(z_2)$ to themselves or to each other. That is, they must be primitive vectors in the lattice $\La$ or in the translate $\La + \vv$ where $\vv$ is a choice of vector connecting the two marked points on the torus (defined up to $\La$, so $\La + \vv$ is well-defined). Thus, if we have $\Pi(\om) = (\La, \vv)$ (viewing $\T$ as the space of marked tori, or equivalently, lattices and a choice of vector), we have $$\La^{sc}_{\om} = \La_{prim} \cup ( \La_{prim}  + \vv),$$ where $\La_{prim}$ denotes the set of primitive vectors in $\La$. 

$\T$ is a fiber bundle over the modular surface $SL(2,\R)/SL(2, \Z)$. It can be broken up into a Haar measure on $SL(2, \R)/SL(2,\Z)$ together with Lebesgue measure on the torus fibers. While a lattice $\La$ will never have two points in $T(0, \sigma)$ for $\sigma < <1$, we can construct a positive measure set of pairs $(\La, \vv) \in \T$ so that $\La_{prim} \cup (\La_{prim} + \vv)$ does intersect $T(\sigma)$ (at least) twice for all $\sigma >0$ by considering the lattices $\La$ so that $\La \cap T(\sigma) \neq \emptyset$ and $\vv \in T(\sigma)$ ($\vv \notin \La$). Thus for all $\alpha$ irrational, all $\sigma >0$, we have
$$\bar{\mu}(\om \in \M(\alpha): \La^{sc}_{\om} \cap T(\sigma) \geq 2) >0.$$
\noindent (In fact, we will obtain a set of measure proportional to $\sigma^4$) (see~\cite[Remark 2.3]{MS} for an explicit description of the distribution of gaps for the set $\vv + \La$). Thus, we obtain:

\begin{Cor}\label{cor:main:gap} For all irrational $\alpha$, 
\begin{equation}\label{eq:cor:gap}\lim_{R \rightarrow \infty} R^2 \gamma^{\omega_{\alpha}}(R) = 0.
\end{equation}
\noindent Moreover, for any $\epsilon >0$, the proportion of gaps less than $\epsilon/R^2$ is positive. That is, writing $\Theta^{\om_{\alpha}}_R : = \{0 \le \theta_1 \le \theta_2 \le \ldots \le \theta_n\}$, we have
\begin{equation}\label{eq:cor:proportion}
\lim_{R \rightarrow \infty} \frac{ |\{1 \le i \le \tilde{N}(\omega_{\alpha}, R): (\theta_{i+1} - \theta_i) \le \epsilon/R^2\}|}{\tilde{N}(\omega_{\alpha}, R)} >0
\end{equation}
\end{Cor}

\end{document}

%% file: wedge.pstex_t
\begin{picture}(0,0)%
\includegraphics{wedge.pstex}%
\end{picture}%
\setlength{\unitlength}{3947sp}%
\begingroup\makeatletter\ifx\SetFigFont\undefined%
\gdef\SetFigFont#1#2#3#4#5{%
  \reset@font\fontsize{#1}{#2pt}%
  \fontfamily{#3}\fontseries{#4}\fontshape{#5}%
  \selectfont}%
\fi\endgroup%
\begin{picture}(1936,1934)(608,-2003)
\put(2101,-586){\makebox(0,0)[lb]{\smash{{\SetFigFont{5}{6.0}{\rmdefault}{\mddefault}{\updefault}{\color[rgb]{0,0,0}$A$}%
}}}}
\put(2176,-286){\makebox(0,0)[lb]{\smash{{\SetFigFont{6}{7.2}{\rmdefault}{\mddefault}{\updefault}{\color[rgb]{0,0,0}$\theta+ \frac{\sigma}{R^2}$}%
}}}}
\put(2476,-511){\makebox(0,0)[lb]{\smash{{\SetFigFont{6}{7.2}{\rmdefault}{\mddefault}{\updefault}{\color[rgb]{0,0,0}$\theta-\frac{\sigma}{R^2}$}%
}}}}
\put(1726,-1036){\makebox(0,0)[lb]{\smash{{\SetFigFont{6}{7.2}{\rmdefault}{\mddefault}{\updefault}{\color[rgb]{0,0,0}$cR$}%
}}}}
\put(2101,-811){\makebox(0,0)[lb]{\smash{{\SetFigFont{6}{7.2}{\rmdefault}{\mddefault}{\updefault}{\color[rgb]{0,0,0}$R$}%
}}}}
\end{picture}%

%% file: triangle.pstex_t
\begin{picture}(0,0)%
\includegraphics{triangle.pstex}%
\end{picture}%
\setlength{\unitlength}{3947sp}%
\begingroup\makeatletter\ifx\SetFigFont\undefined%
\gdef\SetFigFont#1#2#3#4#5{%
  \reset@font\fontsize{#1}{#2pt}%
  \fontfamily{#3}\fontseries{#4}\fontshape{#5}%
  \selectfont}%
\fi\endgroup%
\begin{picture}(2724,1899)(739,-1948)
\put(1951,-1111){\makebox(0,0)[b]{\smash{{\SetFigFont{6}{7.2}{\rmdefault}{\mddefault}{\updefault}{\color[rgb]{0,0,0}$R$}%
}}}}
\put(2476,-961){\makebox(0,0)[b]{\smash{{\SetFigFont{6}{7.2}{\rmdefault}{\mddefault}{\updefault}{\color[rgb]{0,0,0}$L$}%
}}}}
\put(1201,-286){\makebox(0,0)[b]{\smash{{\SetFigFont{6}{7.2}{\rmdefault}{\mddefault}{\updefault}{\color[rgb]{0,0,0}$\theta$}%
}}}}
\end{picture}%

%% file: thin_neck.pstex_t
\begin{picture}(0,0)%
\includegraphics{thin_neck.pstex}%
\end{picture}%
\setlength{\unitlength}{3947sp}%
\begingroup\makeatletter\ifx\SetFigFont\undefined%
\gdef\SetFigFont#1#2#3#4#5{%
  \reset@font\fontsize{#1}{#2pt}%
  \fontfamily{#3}\fontseries{#4}\fontshape{#5}%
  \selectfont}%
\fi\endgroup%
\begin{picture}(1752,579)(4861,-2503)
\put(4876,-2236){\makebox(0,0)[b]{\smash{{\SetFigFont{6}{7.2}{\rmdefault}{\mddefault}{\updefault}{\color[rgb]{0,0,0}$\vv$}%
}}}}
\put(5626,-2011){\makebox(0,0)[b]{\smash{{\SetFigFont{6}{7.2}{\rmdefault}{\mddefault}{\updefault}{\color[rgb]{0,0,0}$\ww$}%
}}}}
\put(5626,-2461){\makebox(0,0)[b]{\smash{{\SetFigFont{6}{7.2}{\rmdefault}{\mddefault}{\updefault}{\color[rgb]{0,0,0}$\ww$}%
}}}}
\end{picture}%

%% file: almost_torus.pstex_t
\begin{picture}(0,0)%
\includegraphics{almost_torus.pstex}%
\end{picture}%
\setlength{\unitlength}{3947sp}%
\begingroup\makeatletter\ifx\SetFigFont\undefined%
\gdef\SetFigFont#1#2#3#4#5{%
  \reset@font\fontsize{#1}{#2pt}%
  \fontfamily{#3}\fontseries{#4}\fontshape{#5}%
  \selectfont}%
\fi\endgroup%
\begin{picture}(2655,2737)(3736,-4973)
\put(4951,-2311){\makebox(0,0)[b]{\smash{{\SetFigFont{5}{6.0}{\rmdefault}{\mddefault}{\updefault}{\color[rgb]{0,0,0}$a$}%
}}}}
\put(6376,-3436){\makebox(0,0)[b]{\smash{{\SetFigFont{5}{6.0}{\rmdefault}{\mddefault}{\updefault}{\color[rgb]{0,0,0}$b$}%
}}}}
\put(5026,-4936){\makebox(0,0)[b]{\smash{{\SetFigFont{5}{6.0}{\rmdefault}{\mddefault}{\updefault}{\color[rgb]{0,0,0}$a$}%
}}}}
\put(3751,-3436){\makebox(0,0)[b]{\smash{{\SetFigFont{5}{6.0}{\rmdefault}{\mddefault}{\updefault}{\color[rgb]{0,0,0}$b$}%
}}}}
\put(5176,-3586){\makebox(0,0)[b]{\smash{{\SetFigFont{5}{6.0}{\rmdefault}{\mddefault}{\updefault}{\color[rgb]{0,0,0}$\le s$}%
}}}}
\end{picture}%

%% file: developwall.pstex_t
\begin{picture}(0,0)%
\includegraphics{developwall.pstex}%
\end{picture}%
\setlength{\unitlength}{4144sp}%
\begingroup\makeatletter\ifx\SetFigFont\undefined%
\gdef\SetFigFont#1#2#3#4#5{%
  \reset@font\fontsize{#1}{#2pt}%
  \fontfamily{#3}\fontseries{#4}\fontshape{#5}%
  \selectfont}%
\fi\endgroup%
\begin{picture}(5353,2300)(870,-2593)
\put(1351,-2131){\makebox(0,0)[lb]{\smash{{\SetFigFont{10}{12.0}{\rmdefault}{\mddefault}{\updefault}{\color[rgb]{0,0,0}$1/2$}%
}}}}
\put(1126,-1141){\makebox(0,0)[lb]{\smash{{\SetFigFont{10}{12.0}{\rmdefault}{\mddefault}{\itdefault}{\color[rgb]{0,0,0}$\alpha$}%
}}}}
\end{picture}%

%% file: gaps32011.bbl
\begin{thebibliography}{99}

\bibitem{BCZ1} F.~Boca, C.~Cobeli, and A.~Zaharescu, \emph{Distribution of Lattice Points Visible from the Origin}, Commun. Math. Phys. 213, 433--470 (2000) 

\bibitem{BPPZ} F.~Boca, V.~Pasol,  A.~Popa, and A.~Zaharescu, \emph{Pair correlation of angles between reciprocal geodesics on the modular surface}, preprint arXiv:1102.032

\bibitem{CHM} Y. ~Cheung, P. ~Hubert, H. ~Masur, \emph{Topological dichotomy and strict ergodicity for translation surfaces}  Ergodic Theory Dynam. Systems  28  (2008),  no. 6, 1729--1748.

\bibitem{EMM} A.~Eskin, J.~Marlkof, D.~Morris, \emph{Unipotent flows on the space of branched covers of Veech surfaces.} Ergodic Theory Dynam. Systems 26 (2006), no. 1, 129--162. 

\bibitem{EskinMasur} A.~Eskin and H.~Masur, \emph{Asymptotic Formulas on Flat Surfaces}, Ergodic Theory and Dynam. Systems, v.21, 443-478, 2001. 

\bibitem{EMS} A.~Eskin, H.~Masur, M.~Schmoll, \emph{Billiards in rectangles with barriers.} Duke Math. J. v.118 (2003), no. 3, 427--463.

\bibitem{EMZ} A.~Eskin, H.~Masur, and A.~Zorich, \emph{Moduli spaces of abelian differentials: the principal boundary, counting problems, and the Siegel-Veech constants.} Publ. Math. Inst. Hautes Etudes Sci. No. 97 (2003), 61--179.

\bibitem{KZ} M.~Kontsevich and A.~Zorich, \emph{Connected components of the moduli spaces of Abelian differentials with prescribed singularities.}
Invent. Math. 153 (2003), no. 3, 631--678. 

\bibitem{Lanneau} E.~Lanneau, \emph{Connected components of the strata of the moduli spaces of quadratic differentials.}
Ann. Sci. Ec. Norm. Super. (4) 41 (2008), no. 1, 1--56. 


\bibitem{Marklof} J.~Marklof, \emph{Distribution modulo one and Ratner's theorem. Equidistribution in number theory, an introduction.} 217--244, 
NATO Sci. Ser. II Math. Phys. Chem., 237, Springer, Dordrecht, 2007. 

\bibitem{MS} J.~Marklof and A.~Strombergsson, \emph{Distribution of free path lengths in the periodic Lorentz gas and related lattice point problems.} preprint. arxiv:0706.4395 (math.DS)

\bibitem{Masur:IET} H.~Masur, \emph{Interval exchange transformations and measured foliations.} 
Ann. of Math. (2) 115 (1982), no. 1, 169--200. 

\bibitem{Masur:billiards} H.~Masur, \emph{Closed trajectories for quadratic differentials with an application to billiards.} Duke Math. J. 53 (1986), no. 2, 307--314. 

\bibitem{Masur} H.~Masur, \emph{The growth rate of trajectories of a quadratic differential},
Ergodic Theory Dynam. Systems 10 (1990), no. 1, 151-176. 

\bibitem{MasurSmillie} H.~Masur and J.~Smillie, \emph{Hausdorff dimension of sets of nonergodic measured foliations.} Ann. of Math. (2) 134 (1991), no. 3, 455--543.

\bibitem{MasurTab}  H.~Masur and S.~Tabachnikov, \emph{Rational Billiards and
Flat Structures}, Handoook of Dynamical Systems, v. 1A, 1015-1089, 2002.

\bibitem{Shah:SL2} N.~Shah,
\emph{Limit distributions of expanding translates of certain
orbits on homogeneous spaces.} Proc. Indian Acad. Sci. (Math Sci) 106(2),
(1996), pp. 105--125.

\bibitem{nosmalltri} J. ~Smillie and B. ~Weiss, \emph{Characterizations of lattice surfaces},  Invent. Math.  180  (2010),  no. 3, 535--557. 

\bibitem{Veech:gauss} W.~Veech, \emph{Gauss measures for transformations on the space of interval exchange maps.} Ann. of Math. (2) 115 (1982), no. 1, 201--242.

\bibitem{hyp ellip} W. ~Veech, \emph{Geometric realizations of hyperelliptic curves}  Algorithms, fractals, and dynamics (Okayama/Kyoto, 1992),  217--226.

\bibitem{Veech} W.~Veech,  \emph{Siegel measures}. Ann. of Math. (2) 148 (1998), no. 3, 895-944. 


\bibitem{Vorobets96} Y.~Vorobets, \emph{Planar structures and billiards in rational polygons: the Veech alternative}, Russian Mathematical Surveys, 51, no. 5. 1996, 779-817.

\bibitem{Vorobets} Y.~Vorobets, \emph{Periodic geodesics on generic translation surfaces.}, Algebraic and topological dynamics, 205--258, Contemp. Math., 385, Amer. Math. Soc., Providence, RI, 2005. 

\bibitem{Zelmjakov:Katok} A.~Zemljakov and A.~Katok, \emph{Topological transitivity of billiards in polygons.} (Russian) Mat. Zametki 18 (1975), no. 2, 291--300.

\bibitem{Zorich:survey} A.~Zorich, \emph{Flat surfaces.} Frontiers in number theory, physics, and geometry. I, 437--583, Springer, Berlin, 2006

\end{thebibliography}
